\documentclass[11pt]{amsart}
\usepackage{amsmath,amssymb,graphicx,amsthm,amsfonts,verbatim,mathtools,thmtools,tikz-cd,fullpage}
\usepackage[shortlabels]{enumitem}
\usepackage[all,cmtip]{xy}
\usepackage{cleveref}
\newtheorem{thm}{Theorem}[section]
\newtheorem{lem}[thm]{Lemma}
\newtheorem{cor}[thm]{Corollary}
\newtheorem{prop}[thm]{Proposition}
\theoremstyle{definition}
\newtheorem{defn}[thm]{Definition}

\newtheorem{exmp}[thm]{Example}

\newcommand{\iso}{\cong}

\newcommand{\Me}{\overline{\mathcal{M}}_{1,1}}

\newcommand{\Xf}{\mathfrak{X}}
\newcommand{\Mf}{\mathfrak{M}}

\newcommand{\Zf}{\mathfrak{Z}}
\newcommand{\Ef}{\mathfrak{E}}
\newcommand{\Yf}{\mathfrak{Y}}

\newcommand{\DM}{\mathbf{DM}(K,\mathbb{Q})}

\newcommand{\DTM}{\mathbf{DTM}(K,\mathbb{Q})}

\newcommand{\Ql}{\Qb_{\ell}}

\newcommand{\Dc}{\mathcal{D}}

\newcommand{\Oc}{\mathcal{O}}
\newcommand{\Pc}{\mathcal{P}}
\newcommand{\Lc}{\mathcal{L}}

\newcommand{\Q}{\mathbb{Q}}
\newcommand{\N}{\mathbb{N}}

\newcommand{\Pb}{\mathbb{P}}
\newcommand{\Ab}{\mathbb{A}}
\newcommand{\A}{\mathbb{A}}

\newcommand{\Fb}{\mathbb{F}}
\newcommand{\Gb}{\mathbb{G}}
\newcommand{\Lb}{\mathbb{L}}

\newcommand{\Qb}{\mathbb{Q}}

\newcommand{\Zb}{\mathbb{Z}}
\newcommand{\Fqbar}{\overline{\mathbb{F}}_q}

\DeclareMathOperator{\Sym}{Sym}

\DeclareMathOperator{\Hom}{Hom}

\begin{document}

    \title[]
    {Motive of the moduli stack of rational \\ curves on a weighted projective stack}
    
    \author{Jun--Yong Park and Hunter Spink}
    
    \address{Center for Geometry and Physics, Institute for Basic Science (IBS), Pohang 37673, Korea}
    \email{junepark@ibs.re.kr}

    \address{Department of Mathematics, Stanford University, Stanford, CA 94305, USA}
    \email{hspink@stanford.edu}

    \begin{abstract}    

    We show the compactly supported motive of the moduli stack of degree $n$ rational curves on the weighted projective stack $\mathcal{P}(a,b)$ is of mixed Tate type over any base field $K$ with $\text{char}(K) \nmid a,b$ and has class  $\mathbb{L}^{(a+b)n+1}-\mathbb{L}^{(a+b)n-1}$ in the Grothendieck ring of stacks. In particular, this improves upon the result of \cite{HP} regarding the arithmetic invariant of the moduli stack $\mathcal{L}_{1,12n} \coloneqq \mathrm{Hom}_{n}(\mathbb{P}^1, \overline{\mathcal{M}}_{1,1})$ of stable elliptic fibrations over $\mathbb{P}^{1}$ with $12n$ nodal singular fibers and a marked Weierstrass section.

    \end{abstract}
    
    \maketitle


    \section{Introduction}\label{sec:intro}

     Fix a base field $K$, and define for $a,b\in \N$ the  weighted projective stack $\Pc(a,b)\coloneqq[(\Ab_{x,y}^2\setminus 0)/\Gb_m]$ where $\lambda \in \Gb_m$ acts by $\lambda \cdot (x,y)=(\lambda^a x, \lambda^b y)$. In this paper, we consider the Voevodsky mixed motive of the Hom stack $\Hom_{n}(\Pb^{1}, \Pc(a,b))$, parameterizing degree $n$ morphisms $f:\Pb^1 \rightarrow \Pc(a,b)$. Of special interest is when $(a,b)=(4,6)$ and $\text{char}(K)\ne 2,3$: we then have $\Pc(4,6)=\overline{\mathcal{M}}_{1,1}$ so this Hom stack is equal to $\mathcal{L}_{1,12n}\coloneqq\Hom_n(\mathbb{P}^1,\overline{\mathcal{M}}_{1,1})$, the moduli stack of stable elliptic fibrations over $\mathbb{P}^1$ with $12n$ nodal singular fibers and a marked Weierstrass section (see \cite{HP}). For this particular case, \cite{HP} obtained the class in $K_0(\mathrm{Stck}_{/K})$ (the Grothendieck ring of stacks over $K$) as $\Lb^{10n+1}-\Lb^{10n-1}$ when $\text{char}(K)=0$,  where $\mathbb{L}\coloneqq[\mathbb{A}^1]$ denotes the Lefschetz motive, and the weighted $\mathbb{F}_q$-point count as $q^{10n+1}-q^{10n-1}$ when $2,3\nmid q$. In particular, $\mathcal{L}_{1,12n}$ is of ``Tate type'' in   the Grothendieck ring of stacks over $K$ for $\text{char}(K)=0$, i.e., it is a polynomial in $\mathbb{L}$. In \cite{Park}, the $\ell$-adic \'etale cohomology and Galois representations of $\mathcal{L}_{1,12n}$ were computed, showing the moduli stack $\Lc_{1,12n}$ is also of ``Tate type'' in the corresponding category.

     Our main theorem unifies and generalizes these results, showing the compactly supported motive $\Mf^{c}(\Hom_{n}(\Pb^{1}, \Pc(a,b)))$ when $\text{char}(K) \nmid a,b$ in Voevodsky's triangulated category $\DM$ of mixed motives with $\mathbb{Q}$-coefficients lies in the subcategory $\DTM$ of mixed Tate motives.

    \begin{thm} \label{thm:Tatemotivic}
    Let $K$ be a field with $\mathrm{char}(K) \nmid a,b$. Then, $\Hom_n(\mathbb{P}^1,\mathcal{P}(a,b))$ is of Tate type in $K_0(\mathrm{Stck}_{/K})$ and the compactly supported motive is of mixed Tate type, i.e.
        \[ \Mf^{c}\left(\Hom_{n}(\Pb^{1}, \Pc(a,b))\right) \in \DTM \subset \DM.\]
    \end{thm}
    Although $\Xf\in \DTM$ does not imply $\Xf$ is of Tate type in $K_0(\mathrm{Stck}_{/K})$, the same proof adapts to both results. This yields the following extension of  \cite[Theorem 1]{HP} for $\text{char}(K)>0$.
    
    \begin{cor}
        \label{thm:motivecount}
        If $\mathrm{char}(K)\nmid a,b$, then 
        \[[\Hom_n(\Pb^1,\Pc(a,b))] = \Lb^{(a+b)n+1}-\Lb^{(a+b)n-1}\;\in K_0(\mathrm{Stck}_{/K}),\]
        and if $\mathrm{char}(\Fb_q) \nmid a,b$, then we have the weighted $\mathbb{F}_q$-point count 
        $$|\Hom_n(\mathbb{P}^1,\mathcal{P}(a,b))(\mathbb{F}_q)|\coloneqq\sum_{x\in \Hom_n(\mathbb{P}^1,\mathcal{P}(a,b))(\mathbb{F}_q)} \frac{1}{|{\textup{Aut}}(x)|} =q^{(a+b)n+1}-q^{(a+b)n-1}.$$
    \end{cor}

   In the above corollary, because we have the class of $\Hom_n(\mathbb{P}^1,\mathcal{P}(a,b))\in K_0(\mathrm{Stck}_{/\Fb_q})$,  we are able to directly acquire the weighted point count over $\Fb_q$ with $\mathrm{char}(\Fb_q) \neq 2,3$ via the point counting measure $\#_q: K_0(\mathrm{Stck}_{/\Fb_q}) \rightarrow \Q$ (c.f. \cite[\S 2]{Ekedahl}).

   

 The following diagram with $\Xf=\Hom_n(\mathbb{P}^1,\mathcal{P}(a,b))$ summarizes the above discussion.

    
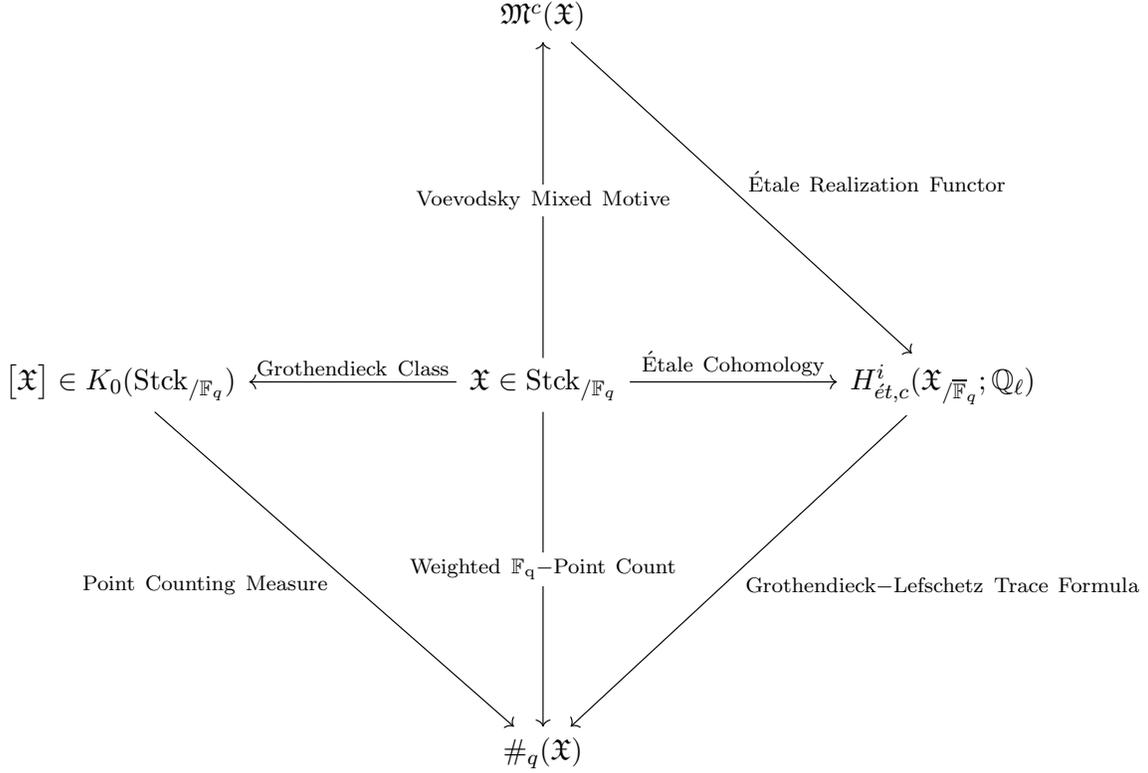
\begin{center}
        \begin{tikzcd}
            &  &  & {\Mf^{c}(\Xf)} \arrow[rrrdddddd, "\mathrm{\acute{E}tale~Realization~Functor}"] &  &  &  \\
            &  &  &  &  &  &  \\
            &  &  &  &  &  &  \\
            &  &  &  &  &  &  \\
            &  &  &  &  &  &  \\
            &  &  &  &  &  &  \\
            \big[\Xf\big] \in K_0(\mathrm{Stck}_{/\Fb_q}) \arrow[rrrdddddd, "\mathrm{Point~Counting~Measure}"'] &  &  & \mathrm{\Xf} \in \mathrm{Stck}_{/\Fb_q} \arrow[rrr, "\mathrm{\acute{E}tale~Cohomology}"] \arrow[lll, "\mathrm{Grothendieck~Class}"'] \arrow[dddddd, "\mathrm{Weighted~\Fb_q-Point~Count}" description] \arrow[uuuuuu, "\mathrm{Voevodsky~Mixed~Motive}" description] &  &  & H^i_{\acute{e}t,c}({\Xf}_{/\Fqbar};\Ql) \arrow[llldddddd, "\mathrm{Grothendieck-Lefschetz~Trace~Formula}"] \\ 
            &  &  &  &  &  &  \\
            &  &  &  &  &  &  \\
            &  &  &  &  &  &  \\
            &  &  &  &  &  &  \\
            &  &  &  &  &  &  \\
            &  &  & \#_q(\Xf) &  &  & 
        \end{tikzcd}
    \end{center}

    The main technical result we prove establishes a certain morphism $\Psi$ considered in Step 2 of the proof of \cite[Proposition 18]{HP} (see also \cite[Equation (3.3)]{FW}) is a piecewise isomorphism, with pieces products of $\mathbb{G}_m$'s and $\mathbb{A}^1$'s. We remark that this partially fixes the argument in \cite[Proof of Theorem 1.2]{FW} (see the corrigendum \cite{FW2}), which was then used in \cite[Proof of Lemma 6.2]{Horel}.
    
    Recall that we have the canonical identification $ \mathbb{A}^m \cong \text{Sym}^m(\mathbb{A}^1)$ by identifying $(a_1,\ldots,a_m)$ with the (unordered collection of roots of the) monic polynomial $f(z)=z^m+a_1z^{m-1}+\ldots+a_m$. Under this identification, consider the morphism
    \begin{align*}\overline{\Psi}:\Ab^{d_1-k} \times \Ab^{d_2-k} \times \Ab^k &\rightarrow \Ab^{d_1}\times \mathbb{A}^{d_2}\\
    \overline{\Psi}(f_1,f_2,g)&=(f_1g,f_2g).
    \end{align*} 

    \par Let $\mathrm{Poly}_1^{(d_1-k,d_2-k)} \subset \Ab^{d_1-k} \times \Ab^{d_2-k}$ be the open subvariety parameterizing the pairs of monic coprime polynomials of degrees $(d_1-k,d_2-k)$, and let $R_{1,\ell}^{(d_1,d_2)}\subset \Ab^{d_1+d_2}$ be the closed subvariety of monic polynomials which share a common factor of degree $\ell$.

    \begin{prop}\label{prop:3.3isom}
    There is a locally closed decomposition of $\textup{Poly}_1^{(d_1-k,d_2-k)}=\bigsqcup C_i$ where each $C_i$ is isomorphic to a product of $\mathbb{A}^1$'s and $\mathbb{G}_m$'s, such that the surjective morphism
     \[\Psi: \mathrm{Poly}_1^{(d_1-k,d_2-k)} \times \Ab^k \rightarrow R_{1,k}^{(d_1,d_2)}\setminus R_{1,k+1}^{(d_1,d_2)},\] obtained by restricting $\overline{\Psi}$, is an isomorphism when restricted to each $C_i\times \mathbb{A}^k$.
    \end{prop}
    In fact, we are able to show (\Cref{prop:MoreGeneral}) a more general result concerning $m$-tuples $(f_1,\ldots,f_m)$. To do this, we decompose the domain according to how a running of the Euclidean algorithm would behave, and show on each piece that we can reconstruct $g$ from the coefficients of $f_1g,\ldots,f_mg$ in a polynomial fashion (although different pieces may involve different reconstructions).

    \section{Definitions}\label{sec:ModuliofRationalCurves}

    \par We first recall some definitions concerning the stack $\Pc(a,b)$.

    \begin{defn}\label{def:wtproj}
        For $a,b\in \N$, the weighted projective stack $\mathcal{P}(a,b)$ is defined as the quotient stack
        \[
        \Pc(a,b)\coloneqq[(\Ab_{x,y}^2\setminus 0)/\Gb_m],
        \]
        where $\lambda \in \Gb_m$ acts by $\lambda \cdot (x,y)=(\lambda^a x, \lambda^b y)$. In this case, $x$ and $y$ have degrees $a$ and $b$, respectively. The line bundle $\Oc_{\Pc(a,b)}(m)$ is defined to be the line bundle associated with the sheaf of degree $m$ homogeneous rational functions without poles on $\Ab^2_{x,y}\setminus 0$. We say that a map $\mathbb{P}^1\to \mathcal{P}(a,b)$ has degree $n$ if $f^*\mathcal{O}_{\mathcal{P}(a,b)}(1)\cong \mathcal{O}_{\mathbb{P}^1}(n)$.
    \end{defn}
    
    \begin{exmp}
        When $\text{char}(K)\ne 2,3$, \cite[Proposition 3.6]{Hassett} shows that one example is given by the proper Deligne--Mumford stack of stable elliptic curves $(\Me)_K \cong [ (\mathrm{Spec}~K[a_4,a_6]-(0,0)) / \Gb_m ] = \Pc_K(4,6)$ by using the short Weierstrass equation $y^2 = x^3 + a_4x + a_6x$, where $\lambda \cdot a_i=\lambda^i a_i$ for $\lambda \in \Gb_m$ and $i=4,6$. This is no longer true when $\text{char}(K)\in \{2,3\}$, as the Weierstrass equations are more complicated. 
    \end{exmp}
    \begin{defn}
    The stack $\Hom_n(\mathbb{P}^1,\mathcal{P}(a,b))$ is defined to be the Hom stack of degree $n$ morphisms $\mathbb{P}^1\to \mathcal{P}(a,b)$. Concretely, $$\Hom_n(\mathbb{P}^1,\mathcal{P}(a,b))=[T/\mathbb{G}_m]$$ where $T\subset (H^0(\mathcal{O}_{\mathbb{P}^1}(an))\setminus 0)\oplus (H^0(\mathcal{O}_{\mathbb{P}^1}(bn))\setminus 0)$ is the open subset of pairs of nonzero polynomials $(\widetilde{u}(z,t),\widetilde{v}(z,t))$, homogeneous of degrees $an$ and $bn$, respectively, sharing no common factor, and $\lambda \in \mathbb{G}_m$ acts by $\lambda \cdot (\widetilde{u}(z,t),\widetilde{v}(z,t))\coloneqq(\widetilde{u}(\lambda z, \lambda t),\widetilde{v}(\lambda z, \lambda t))=(\lambda^a \cdot \widetilde{u}(z,t),\lambda^b \cdot \widetilde{v}(z,t))$.
    \end{defn}

    \begin{prop}\label{prop:DMstack}
        For $a,b,n \in \N$ and  $\mathrm{char}(K)\nmid a,b$, the Hom stack $\Hom_n(\Pb^1,\Pc(a,b))$ parameterizing degree $n$ morphisms $f:\Pb^1 \rightarrow \Pc(a,b)$ is a smooth separated tame Deligne--Mumford stack of finite type and dimension $(a+b)n+1$.
    \end{prop}

    \begin{proof} 
    This was established in \cite[Proposition 9, Proof of Theorem 1]{HP}. To recall the major points therein, $\Hom_n(\Pb^1,\Pc(a,b))=[T/\mathbb{G}_m]$ is a smooth Deligne--Mumford stack by \cite[Theorem 1.1]{Olsson}, admitting $T$ as a smooth schematic cover.
    As $\Gb_m$ acts on $T$ properly with positive weights $a,b\in \N$, the quotient stack $[T/\Gb_m]$ is separated. It is tame by \cite[Theorem 3.2]{AOV} since $\mathrm{char}(K)\nmid a,b$.
    \end{proof}
    As a polynomial in $H^0(\mathcal{O}_{\mathbb{P}^1}(an)) \setminus 0$ or $H^0(\mathcal{O}_{\mathbb{P}^1}(bn)) \setminus 0$ is determined by its restriction away from $\infty=[0:1]$ (i.e. by evaluating $t=1$), we may consider $T$ as the set of not necessarily monic pairs of nonzero coprime polynomials $(u(z),v(z))$ such that either $\deg u= an$ and $0\le \deg v \le bn$ or $\deg v=bn$ and $0\le \deg u \le an$. Here the degree conditions are precisely ensuring that $\widetilde{u},\widetilde{v}$ do not share a common zero at $\infty$.

\begin{defn}
    For $(k,\ell)\in \{(an,0),\ldots, (an,bn-1),(0,bn),\ldots,(an-1,bn),(an,bn)\}$, we define 
    $$T_{k,l}\coloneqq\{(u,v) \in T : \deg u=k, \; \deg v=l \}\subset T.$$
\end{defn}
Notice that $\overline{T_{k,bn}}= \bigsqcup_{0\le k' \le k} T_{k',bn}$, and $\overline{T_{an,\ell}}=\bigsqcup_{0 \le \ell'\le \ell} T_{an,\ell'}$, so we obtain the stratification
    
    \begin{align}
    T&=T_{an,bn} \sqcup \left(\bigsqcup_{k=0}^{an-1} T_{k,bn}\right) \sqcup \left(\bigsqcup_{l=0}^{bn-1} T_{an,l} \right)
    \end{align}
    \begin{align*}
    T&=\overline{T_{an,bn}} \supsetneq \overline{T_{an-1,bn}} \supsetneq \dotsb \supsetneq \overline{T_{0,bn}}=T_{0,bn}\\
    T&=\overline{T_{an,bn}} \supsetneq \overline{T_{an,bn-1}} \supsetneq \dotsb \supsetneq \overline{T_{an,0}}=T_{an,0}
    \end{align*}
    \[\overline{T_{an-k,bn}} \cap \overline{T_{an,bn-l}} =\emptyset ~~\; \forall k,l>0. \]
    \begin{defn}\label{def:poly}
    For any $k,\ell \ge 0$, define \[\text{Poly}_1^{k,l}\coloneqq\{(u(z),v(z)):  u(z),v(z) \text{ are monic and coprime of degrees $k,\ell$ respectively} \}.\]
    Equivalently, $\text{Poly}_1^{(k,l)}$ is the complement of the resultant hypersurface $\mathcal{R}^{(k,\ell)}\subset\Sym^k \mathbb{A}^1 \times \Sym ^\ell \mathbb{A}^1=\mathbb{A}^k\times \mathbb{A}^\ell$ of pairs of monic polynomials $(u(z),v(z))$ of degrees $k,\ell$  respectively which share a factor.
    \end{defn}
   Note that when $T_{k,\ell}$ is defined, we have $\mathbb{G}_m\times \mathbb{G}_m \times \text{Poly}_1^{(k,\ell)} \cong T_{k,\ell}$, taking $(\lambda_1,\lambda_2,u,v)\mapsto (\lambda_1u,\lambda_2v)$. The arithmetic properties of $\mathrm{Poly}_1^{(k,l)}$ were studied in \cite{HP}, inspired by \cite{Segal, FW}.

    \begin{defn}\label{def:filtration}
        Let $R_{1,k}^{(d_1,d_2)}\subset \A^{d_1}\times \A^{d_2}$ be the pairs $(u,v)$ of monic polynomials of degree $d_1,d_2$ respectively for which there exists a common factor of degree $\ge k$. Alternately, $R_{1,k}^{(d_1,d_2)}$ is the loci where the Sylvester matrix, whose determinant computes the resultant of $u$ and $v$, is of rank $\le d_1+d_2-k$.
    \end{defn}  
    
    The extraneous index $1$ is kept for consistency with \cite{FW}. Note that $R_{1,0}^{(d_1,d_2)} = \A^{d_1}\times \A^{d_2}$ and $R_{1,1}^{(d_1,d_2)}=\mathcal{R}^{(d_1,d_2)}$, and the $R_{1,k}^{(d_1,d_2)}$ give a descending filtration of closed subvarieties
    \[ \A^{d_1}\times \A^{d_2} = R_{1,0}^{(d_1,d_2)} \supset R_{1,1}^{(d_1,d_2)} \supset \cdots \supset R_{1,\min(d_1,d_2)+1}^{(d_1,d_2)}=\emptyset.    \] 
    
    \section{Proof of \Cref{prop:3.3isom}}
    We will in fact prove a more general statement concerning $m$-tuples $f_1(z),\ldots,f_m(z)$ of polynomials of degrees $d_1-k,\ldots,d_m-k$, with the $m=2$ case specializing to \Cref{prop:3.3isom}. Recall that $\Sym^{d}\mathbb{A}^1=\mathbb{A}^{d}$ parametrizes monic polynomials of degree $d_i$, so we may say $f_i\in\mathbb{A}^{d_i-k}$.
    
    Let $\text{Poly}_1^{(d_1-k,\ldots,d_m-k)}\subset \mathbb{A}^{d_1-k}\times \ldots \times \mathbb{A}^{d_m-k}$ be the open subvariety where the $f_i$ do not all share a common factor. Concretely, this is where the Sylvester matrix of $f_1,\ldots,f_m$ has rank at most $(d_1-k)+\ldots+(d_m-k)-m+1$. Let $R_{1,\ell}^{(d_1,\ldots,d_m)}\subset \mathbb{A}^{d_1}\times \ldots \times \mathbb{A}^{d_m}$ be the subvariety of tuples of monic polynomials $(h_1,\ldots,h_m)$ of degrees $d_1,\ldots,d_m$ which share a common factor of degree $\ge \ell$. Concretely, this is where the Sylvester matrix of $h_1,\ldots,h_m$ has rank at most $d_1+\ldots+d_m-\ell (m-1)$.
    \begin{prop}\label{prop:MoreGeneral}
    There is a locally closed decomposition $\textup{Poly}_1^{(d_1-k,\ldots,d_m-k)}=\bigsqcup C_i$ where each $C_i$ is isomorphic to a product of $\mathbb{A}^1$'s and $\mathbb{G}_m$'s, such that the surjective morphism
    \begin{align*}\Psi:\textup{Poly}_1^{(d_1-k,\ldots,d_m-k)}\times \mathbb{A}^k &\to R^{(d_1,\ldots,d_m)}_{1,k}\setminus R_{1,k+1}^{(d_1,\ldots,d_m)}\\
    (f_1,\ldots,f_m,g)&\mapsto (f_1g,\ldots,f_mg)\end{align*}
    is an isomorphism when restricted to each $C_i\times \mathbb{A}^k$.
    \end{prop}
    \begin{proof}
    For a closed point $x\in \text{Poly}_1^{(d_1-k,\ldots,d_m-k)}$, consider the restrictions $\overline{f_1}(z),\ldots,\overline{f_m}(z)$ to $\kappa_x[z]$, where $\kappa_x$ is the residue field at $x$. Recall the Euclidean algorithm runs as follows. Set $(\overline{f_i})_0=\overline{f_i}$ for all $i$. Then assuming $(\overline{f_i})_j$ are all defined for some $j$, let $c_j$ be an index such that $(\overline{f_{c_j}})_j$ has smallest degree among all  nonzero $(\overline{f_{i}})_j$.
    Set $(\overline{f_{c_j}})_{j+1}=(\overline{f_{c_j}})_{j}$, and for $i\ne c_j$ set $$(\overline{f_i})_{j+1}=\frac{1}{\alpha_{i,j+1}}((\overline{f_i})_{j}-(\overline{g_i})_{j+1}(\overline{f_{c_j}})_{j})$$
    where $(\overline{g_i})_{j+1}\in \kappa_x[z]$ is uniquely chosen so that the degree of the right hand side is smaller than the degree of $(\overline{f_{c_j}})_j$, and $\alpha_{i,j+1}\in \kappa_x$ is chosen so that $(\overline{f_i})_{j+1}$ is monic (set $\alpha_{i,j+1}=1$ if $(\overline{f_i})_{j}-(\overline{g_i})_{j+1}(\overline{f_{c_k}})_{j}=0$). Eventually, for some $\ell$ we will have exactly one nonzero  $(\overline{f_i})_{\ell}$, which will be the monic greatest common divisor of $\overline{f_1}(z),\ldots,\overline{f_m}(z)$ (which as $x\in \text{Poly}_1^{(d_1-k,\ldots,d_m-k)}$ will be equal to $1$).
    
    Construct a locally closed subset $x\in C_x \subset \text{Poly}^{(d_1-k,\ldots,d_m-k)}$ as follows. Set $C_{x,0}=\mathbb{A}^{d_1-k}\times \ldots \times \mathbb{A}^{d_m-k}$, and set $(f_i)_0=f_i$ for all $i$. Suppose $C_{x,j}$ and $(f_i)_j$ are constructed so that the reduction in $\kappa_x[z]$ is $(\overline{f_i})_j$ and $\deg ((f_i)_j)=\deg ((\overline{f_i})_j)$ for all $i$. Then we construct $C_{x,j+1}\subset C_{x,j}$ as follows. There are unique polynomials $(g_i)_{j+1}$ such that $\deg ((f_i)_j-(g_i)_{j+1}(f_{c_j})_j) < \deg (f_{c_j})_j$. These reduce to the $(\overline{g_i})_{j+1}$. Let $C_{x,j+1}$ be the subset of $C_{x,j}$ where the $z^i$ coefficient of $((f_i)_j-(g_i)_{j+1}(f_{c_k})_j)$ is $0$ for all $i>\deg (\overline{f_i})_{j+1}$, and where the $z^i$ coefficient is nonzero for $i=\deg (\overline{f_i})_{j+1}$. The $C_{x,j}$ are constructed so that the Euclidean algorithm may be run with $f_i$ identically to how it was run for $\overline{f_i}$, so eventually we will conclude the Euclidean algorithm and for some $\ell$ exactly one $(f_i)_\ell$ will be equal $1$ and the rest will be zero. In particular, this implies that $(f_1,\ldots,f_m)\in C_{x,\ell}$, have no common zero at any closed point, so $C_{x,\ell}\subset \text{Poly}_1^{(d_1-k,\ldots,d_m-k)}$. We let $C_x=C_{x,\ell}$.
    
    The locally closed set $C_x$ is entirely determined by the degrees of all of the $(\overline{f_i})_j$ for all $i,j$, so there are finitely many, and they partition $\text{Poly}_1^{(d_1-k,\ldots,d_m-k)}$. Furthermore, running the Euclidean algorithm backwards, it is clear that $C_{x,j}=C_{x,j+1}\times\prod_{(f_i)_{j} \ne 0 \text{ and }i\ne c_j} (\mathbb{G}_m \times \mathbb{A}^{\deg((g_i)_{j+1})})$, corresponding to the free choices of $\alpha_{i,j+1}$ and $(g_{i})_{j+1}(z)$.
    
    Thus to conclude, it suffices to show that $\Psi$ is an isomorphism when restricted to each $C_i\times \mathbb{A}^k$. Indeed, consider the analogously defined partition of $R_{1,k}^{(d_1,\ldots,d_m)}\setminus R_{1,k+1}^{(d_1,\ldots,d_m)}=\bigsqcup C_i'$ according to the intermediate degrees of polynomials in a run of the Euclidean algorithm, ending with all polynomials $0$ except one monic polynomial of degree $k$. Then after possibly reindexing, we have that $\Psi$ maps $C_i\times \mathbb{A}^k \to C_i'$ where the intermediate degrees of the nonzero polynomials for $C_i'$ are all exactly $k$ larger than the corresponding degrees in $C_i$. The existence of an inverse map is equivalent to being able to reconstruct the coefficients of $g$ from the coefficients of $f_1g,\ldots,f_mg$ in a polynomial fashion on $C_i'$, but this is accomplished by the Euclidean algorithm on $C_i'$ by construction.
    \end{proof}

    \section{Class in the Grothendieck ring of $K$-stacks}\label{sec:GrothendieckMotive}
We recall the definition of $K_0(\mathrm{Stck}_{/K})$ the Grothendieck ring of algebraic $K$-stacks.

    \begin{defn}\label{defn:GrothringStck}
        \cite[\S 1]{Ekedahl}
        The \emph{Grothendieck ring $K_0(\mathrm{Stck}_{/K})$ of algebraic stacks of finite type over $K$ all of whose stabilizer group schemes are affine} is a group generated by isomorphism classes of $K$-stacks $[\Xf]$ of finite type, modulo relations:
        \begin{itemize}
            \item $[\Xf]=[\Zf]+[\Xf \setminus \Zf]$ for $\Zf \subset \Xf$ a closed substack,
            \item $[\Ef]=[\Xf \times \Ab^n ]$ for $\Ef$ a vector bundle of rank $n$ on $\Xf$.
        \end{itemize}
        Multiplication on $K_0(\mathrm{Stck}_{/K})$ is induced by $[\Xf][\Yf]\coloneqq[\Xf \times_K \Yf]$. There is a distinguished element $\Lb\coloneqq[\A^1] \in K_0(\mathrm{Stck}_{/K})$, called the \emph{Lefschetz motive}.
    \end{defn}



    \Cref{prop:3.3isom} shows that the morphism considered in \cite[Equation (3.3)]{FW} and \cite[Proposition 18]{HP} extends to an isomorphism over $\Zb$, allowing us to extend the computation of \cite{FW, FW2, HP} to fields of positive characteristic. In particular, the expression of $[\mathrm{Poly}_1^{(d_1,d_2)}]$ found in \cite[Proposition 18]{HP} for $\mathrm{char}(K)=0$ as a polynomial in $\Lb$ is the same for fields of positive characteristic.

    \begin{prop}
        \label{prop:poly}
        Fix $d_1,d_2 \ge 0$.
        \[ [\mathrm{Poly}_1^{(d_1,d_2)}] =
        \begin{cases}
        \Lb^{d_1+d_2}-\Lb^{d_1+d_2-1} & \text{ if } d_1,d_2 >0\;, \\
        \Lb^{d_1+d_2} & \text{ if } d_1=0 \text{ or } d_2=0\;.
        \end{cases}
        \]
    \end{prop}

    \begin{proof}

    The proof is identical to the \cite[Proposition 18]{HP} except that at Step 2, we have the morphism $\Psi$ to be an isomorphism by Proposition~\ref{prop:3.3isom}. 

    \end{proof}

    \begin{proof}[Proof of \Cref{thm:motivecount}]
    The proof is identical to the proof of \cite[Theorem 1, Corollary 2]{HP} except at Step 3, we have the Grothendieck classes of $[\mathrm{Poly}_1^{(d_1,d_2)}]$ over any field $K$ by Proposition~\ref{prop:3.3isom}.
    \end{proof}


    \section{Voevodsky's mixed motives and mixed Tate motives}\label{sec:VoevodskyMotive}

    Let $\mathbf{DM}(K,R)$ be Voevodsky's triangulated category of mixed motives with compact support denoted by $C_{*}^{c} (X)$ in \cite{Voevodsky} for separated schemes $X$ of finite type over a base field $K$ in $R$--coefficients. When $K$ is a perfect field and admits a resolution of singularities the formulation and the properties of $\mathbf{DM}(K,R)$ such as the localization triangle are worked out by \cite{Voevodsky, MVW}. When $K$ is a perfect field and may not admit a resolution of singularities but the exponential characteristic of $K$ is invertible in $R$ similar works on $\mathbf{DM}(K^{\mathrm{perf}},R)$ are done by \cite{Kelly}. And for arbitrary field $K$, we know that the pullback functor $\mathbf{DM}(K,R) \to \mathbf{DM}(K^{\mathrm{perf}},R)$ is an equivalence of categories by \cite[Proposition 8.1]{CD} (see also \cite[Theorem 5.1]{Totaro}).



    For our purposes, we note that the definition of the compactly supported motive in $\Qb$--coefficients extends to quotient stacks $[X/G]$ over a field $K$ based on the construction of \cite[Theorem 8.4]{Totaro}, and agrees with the standard definition in the special case of a quasi-projective scheme.

 We now recall the localization exact triangle which will aid us in showing the mixed Tate property of the compactly supported motives by repeated application of the fullness of $\DTM$ (2-out-of-3 property) applied to the stratification.

    \begin{prop}\label{prop:Gysin}
    For a separated scheme $X$ of finite type over an arbitrary field $K$ and a closed subscheme $Z$ of $X$, there is an exact triangle in $\DM$ called the localization exact triangle
    \begin{equation}
    \label{eqn:Gysinseq}
    \Mf^{c}(Z) \to \Mf^{c}(X) \to \Mf^{c}(X-Z) \to \Mf^{c}(Z)[1]
    \end{equation}
    \end{prop}

    \begin{proof} 
    For the proof see \cite[Proposition 4.1.5, Theorem 4.3.7(3)]{Voevodsky} and \cite[Proposition 5.5.5, Theorem 5.5.14(3)]{Kelly} together with the equivalence \cite[Proposition 8.1]{CD} (see also \cite[Theorem 5.1]{Totaro} and \cite[Page 2097]{Totaro}).
    \end{proof}

    Recall that given a triangulated category $\Dc$, a full subcategory $\Dc^{'}$ is a triangulated subcategory if and only if it is invariant under the shift $T$ of $\Dc$ and for any distinguished triangle $ A \to B \to C \to A[1]$ for $\Dc$ where $A$ and $B$ are in $\Dc^{'}$ there is an isomorphism $C \iso C^{'}$ with $C^{'}$ also in $\Dc^{'}$. A full triangulated subcategory $\Dc^{'} \subset \Dc$ is thick if it is closed under direct sums. We now define Voevodsky's triangulated subcategory $\DTM$ of mixed Tate motives.     

    \begin{defn}\label{def:VoevodskyTateMotive} 
        Let $\DTM$ be Voevodsky's full triangulated thick subcategory of mixed Tate motives generated by the Tate objects $\Qb(n)$.
    \end{defn}

    \begin{exmp} The motive $\Mf^c(\Pb^N)$ is of mixed Tate type, and we in fact have $$
        \Mf^c \left( \Pb^{N} \right) = \bigoplus_{i=0}^{N} \Qb\left( i \right)[2i].$$ More generally, any variety with a locally closed decomposition into pieces isomorphic to $\mathbb{G}_m^a\times (\mathbb{A}^1)^b$ for various $a,b$ is of mixed Tate type by repeated application of \Cref{prop:Gysin} (though the motive may not split like with $\Mf^c(\mathbb{P}^n)$). Such varieties were the subject of \cite{Totaro2}, and we remark that these varieties arise naturally as varieties with an action of a split solvable group with finitely many orbits (as each orbit is of the form $\mathbb{G}_m^a\times (\mathbb{A}^1)^b$).
    \end{exmp}

   The following result makes it easy to show that certain compactly supported motives $\Mf^{c}([X/G])$ of quotient stacks $[X/G]$ over a field $K$ are of mixed Tate type, i.e., $\Mf^{c}([X/G]) \in \DTM$, when $G=\mathrm{GL}(n)$.

    \begin{cor}[Corollary 8.12 of \cite{Totaro}] \label{cor:GmquotientTate} 
        Let $\Xf$ be a quotient stack over $K$. Let $E$ be a principal $\mathrm{GL}(n)$-bundle over $\Xf$ for some $n$, viewed as a stack over $K$, or equivalently $\Xf=[E/\mathrm{GL}(n)]$. Then $E$ is mixed Tate in $\DM$ if and only if $\Xf$ is mixed Tate.
    \end{cor}

   We are now ready to prove Theorem~\ref{thm:Tatemotivic}.

    \subsection{Proof of Theorem~\ref{thm:Tatemotivic}}
    \label{subsec:Tatemotivic}

    \begin{proof} 

    By \Cref{cor:GmquotientTate}, showing  $\Mf^{c}\left(\Hom_{n}(\Pb^{1}, \Pc(a,b))\right)=\Mf^c([T/\mathbb{G}_m])$ over a field $K$ with $\mathrm{char}(K)\nmid a,b$ is mixed Tate is equivalent to showing that $\Mf^c(T)$ is of mixed Tate type. 

    As we have the stratification of $T$ by $T_{k,l}$ as in (1), it suffices to show $\Mf^{c}(T_{k,l})$ is of mixed Tate type by the repeated application of \Cref{prop:Gysin}. As $T_{k,\ell}=\mathbb{G}_m\times \mathbb{G}_m\times \mathrm{Poly}_1^{(k,\ell)}$, by \Cref{cor:GmquotientTate} it suffices to show that $\Mf^{c}(\mathrm{Poly}_1^{(k,\ell)})$ is of mixed Tate type.

    The proof that the compactly supported motive $\Mf^{c}(\mathrm{Poly}_1^{(k,\ell)})$ is of mixed Tate type now follows from \Cref{prop:Gysin} applied repeatedly to the locally closed decomposition furnished by \Cref{prop:3.3isom}, each piece of which is of mixed Tate type.

    \end{proof}


    \textbf{Acknowledgments.} We would like to thank the organizers of the Pacific Institute for the Mathematical Sciences Workshop on Arithmetic Topology at University of British Columbia in June 2019 for their excellent hospitality and inspiring conference where most of this work was done. Special thanks to Dori Bejleri, Benson Farb, Changho Han, Craig Westerland and Jesse Wolfson for many helpful conversations. Jun-Yong Park was supported by IBS-R003-D1, Institute for Basic Science in Korea. Finally, we thank the reviewer for their in-depth and helpful suggestions.


\end{document}